\documentclass[12 pt]{amsart}
\usepackage{graphicx}
\newtheorem{theorem}{Theorem}[section]
\newtheorem{lemma}[theorem]{Lemma}
\newtheorem{proposition}[theorem]{Proposition}

\newtheorem{definition}[theorem]{Definition}

\theoremstyle{definition}
\newtheorem{remark}[theorem]{Remark}
\newtheorem{assumption}[theorem]{Assumption}

\numberwithin{equation}{section}
\newcommand{\CC}{\mathbb{C}}
\newcommand{\RR}{\mathbb{R}}

\newcommand{\ZZ}{\mathbb{Z}}
\newcommand{\DD}{\mathbb{D}}

\newcommand{\NN}{\mathbb{N}}
\newcommand{\bV}{\mathbf{V}}
\newcommand{\bH}{\mathbf{H}}
\newcommand{\bL}{\mathbf{L}}
\newcommand{\bE}{\mathbf{E}}
\newcommand{\bF}{\mathbf{F}}
\newcommand{\bS}{\mathbf{S}}

\newcommand{\cH}{\mathcal{H}}
\newcommand{\cM}{\mathcal{M}}
\newcommand{\cN}{\mathcal{N}}
\newcommand{\cK}{\mathcal{K}}

\DeclareMathOperator{\spa}{span}

\DeclareMathOperator{\ran}{ran}
\DeclareMathOperator{\rank}{rank}
\DeclareMathOperator{\diag}{diag}
\title{Schur  functions on a rhombic lattice}
\author[D.  Alpay]{Daniel Alpay}
\address{(DA)
Faculty of Mathematics, Physics, and Computation\\
Schmid College of Science and Technology\\
Chapman University\\
One University Drive\\
Orange, CA 92866\\
USA}
\email{alpay@chapman.edu}

\author[A. Fuerte Perez]{Angel Fuerte Perez}
\address{(AFP) Department of Mathematics\\ Kansas State University\\ 138 Cardwell Hall\\1228 N.  17th Street\\
  Manhattan, KS 66506\\ USA}
  \email{acfuerte@ksu.edu}

\author[D.  Volok]{Dan Volok}
\address{(DV) Department of Mathematics\\ Kansas State University\\ 138 Cardwell Hall\\1228 N.  17th Street\\
  Manhattan, KS 66506\\ USA}
\email{danvolok@math.ksu.edu}

\begin{document}

\begin{abstract}
  We extend the study of discrete analytic (DA) Schur functions to rhombic lattices, utilizing suitably defined shift operators. There is a number of important differences with the classical case, including eigenvalues of the backward shift operator.
  As an application we solve a basic interpolation problem in a weighted Hardy space of DA functions, introducing a discrete counterpart of the Blaschke factor.

\end{abstract}
\maketitle

\mbox{}\\


\date{today}
\setcounter{section}{-1}
\section{Introduction}
In this paper we continue the investigation of discrete analytic (DA) analogues of the Schur class with applications to interpolation. In previous papers \cite{av2022}, \cite{av2023} DA Schur functions were studied on the integer lattice in the complex plane, and a basic interpolation problem on a rectangular region of the lattice was solved by means of a suitable DA Blaschke product. 
An important role was played by a certain basis of DA polynomials and the associated shift operators. Following the ideas of \cite{aktv}, we now generalize the notion of the shift operators to the setting of DA functions on a rhombic lattice $\Lambda$ in the complex plane $\CC$ -- a monohedral tessellation of $\CC$ with  unit rhombi. We define, under suitable assumptions on the lattice, a forward shift operator that is left-invertible, and choose an appropriate left inverse for the role of the backward shift operator.  It turns out that, unlike the classical case of continuous analytic functions in a neighborhood of the origin, the eigenvalues of the backward shift operator form a complement of a finite non-empty set that corresponds to the set of directions of the edges of the lattice (Theorem \ref{genfunthm}). 

Iterating the forward shift operator, we define a linearly independent family of DA functions that plays the role of the standard polynomial basis $z^n$ in this case. Completions of this basis with respect to an inner product, in which the basis becomes orthogonal, lead to weighted Hardy spaces of DA functions. From the investigation of backward shift invariant spaces of vector-valued DA functions we recover the correspondence between co-isometric operator colligations and contractive (convolution) multiplication operators on the weighted Hardy spaces (see \cite{bbth} for an exposition in the continuous setting). In the process, a suitable DA analogue of the Schur kernel
$$K(z,w)=\dfrac{I-S(z)S(w)^*}{1-z\bar{w}}$$ is identified (Theorem \ref{Schkerthm}). Finally, as an important example, we introduce DA Blaschke factors that are used to solve a basic interpolation problem in the weighted DA Hardy space (Theorem \ref{interpolthm}).

The paper consists of five sections besides this Introduction. Section \ref{prelim} deals with the background on DA functions and elementary properties of the shift operators. DA polynomials and their convolution product are studied in Section \ref{eigen}.  In Section \ref{RKHS-DA} we characterize operator-valued positive kernels that give rise to backward shift invariant reproducing kernel Hilbert spaces (RKHS) of DA functions. Section \ref{Sec-Sch} is devoted to the study of the DA Schur class and, finally, Section \ref{Sec-Bla} describes  the basic interpolation problem in the weighted Hardy space of DA functions and its solution by means of DA Blaschke factors.

\section{Preliminaries}\label{prelim}
In what follows,  $\Lambda$ is a rhombic lattice in the complex plane $\CC$ -- a monohedral tessellation of $\CC$ with  unit rhombi (see Figure \ref{rholat}).

\begin{figure}[h]
 \includegraphics[width=\textwidth]{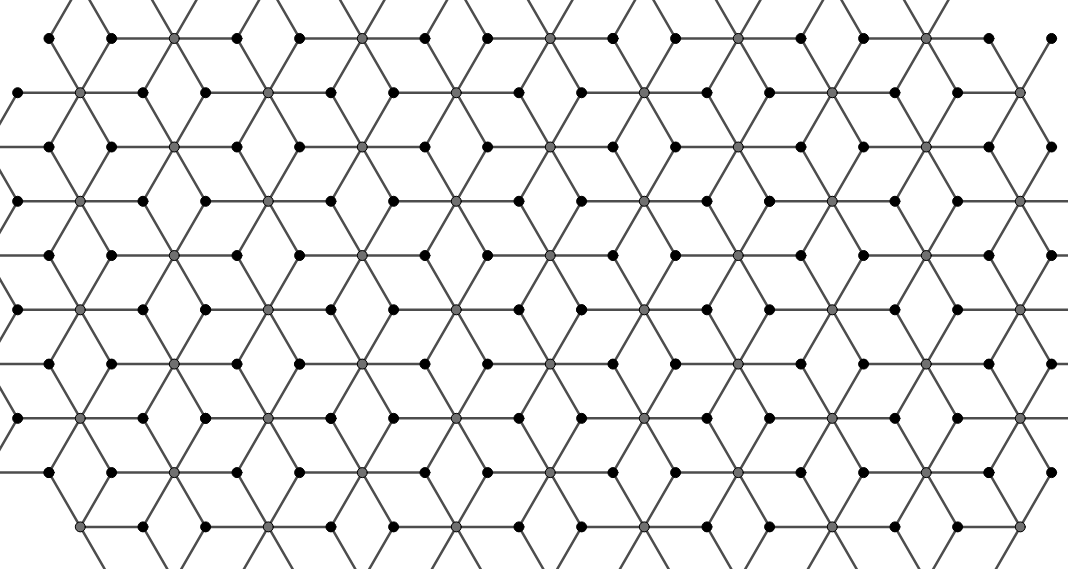}
 \caption{Example of a rhombic lattice}
 \label{rholat}
 \end{figure}

The sets of vertices, edges and faces of $\Lambda$ are denoted by $\bV(\Lambda),$ $\bE(\Lambda)$ and $\bF(\Lambda),$ respectively. The set of directions of the edges of $\Lambda$ is denoted by $\hat{\bE}(\Lambda):$
$$\hat{\bE}(\Lambda)=\{b-a:(a,b)\in \bE(\Lambda)\}.$$ Note that $\hat{\bE}(\Lambda)$ is a finite set (see \cite{aktv}, Proposition 2.1).There is no loss of generality in assuming that $0\in\bV(\Lambda)$ and $1\in\hat{\bE}(\Lambda).$

The first notion to be discussed is that of discrete analyticity. One can find in the literature many different definitions of analyticity for functions on graphs; the definition used here originated with J. Ferrand (\cite{Ferrand}) in the case of a square lattice  and was extended to the case of rhombic lattices by R. J. Duffin in \cite{rhombic}. It can be formulated in the language of discrete integrals.

Let $a,b\in \bV(\Lambda).$ A path in $\Lambda$ from 
$a$ to $b$ is a finite sequence $(z_0,z_1,\dots,z_N)$ of vertices of $\Lambda,$ where $z_0=a,$ $z_N=b,$ and  for $n=1,\dots,N$
the vertices $z_{n-1}$ and $z_n$ are adjacent in $\Lambda.$  A path from $a$ to $b$ is closed if $a=b.$ Given a function $f:\bV(\Lambda)\longrightarrow\mathcal{V},$  where $\mathcal{V}$ is a complex vector space, the discrete integral of $f$ over a path $\gamma=(z_0,\dots,z_N)$ is defined by
$$\int_\gamma f \delta z=\sum_{n=1}^N \dfrac{f(z_{n-1})+f(z_n)}{2}(z_n-z_{n-1}).$$

\begin{definition}\label{ferrand} Function $f:\bV(\Lambda)\longrightarrow\mathcal{V}$ is said to be discrete analytic (DA) if for every closed path $\gamma$ in $\Lambda$
$$\int_\gamma f \delta z=0.$$ The space of DA functions is denoted by $\cH(\Lambda).$
\end{definition}

Note that  every closed path of the form $(z_0,z_1,z_2,z_3,z_0),$
where $z_2\not=z_0$ and $z_3\not=z_1,$ forms a rhombic face of $\Lambda;$ it is denoted simply by $z_0z_1z_2z_3.$

\begin{theorem}[\cite{rhombic}, Theorem 1]\label{CRthm}
Function $f:\bV(\Lambda)\longrightarrow\mathcal{V}$ is DA if, and only if,  on every face $z_0z_1z_2z_3$ of $\Lambda$, $f(z)$ satisfies the discrete Cauchy - Riemann equation
\begin{equation}\label{CR}\dfrac{f(z_0)-f(z_2)}{z_0-z_2}=\dfrac{f(z_1)-f(z_3)}{z_1-z_3}.\end{equation}
\end{theorem}

As is clear from Definition \ref{ferrand},  the integral of a DA function $f(z)$ over a path $\gamma$ from $a$ to $b$ is independent  of the choice of $\gamma,$ hence a simplified notation can be used:
$$\int_a^b f \delta z=\int_\gamma f\delta z,$$
where $\gamma$ is any path in $\Lambda$ from $a$ to $b.$

The following useful observation relies on the fact that the faces of $\Lambda$ are parallelograms.

\begin{theorem}[\cite{rhombic}, Theorem 5] Let $z_0\in \bV(\Lambda)$ be fixed. If $f(z)$ is a DA function, then so is
$$F(z)=\int_{z_0}^z f\delta z.$$\end{theorem}

At this point, one additional assumption about lattice $\Lambda$ needs to be made. This assumption can be stated in the language of tracks, adopted from \cite{kenyon}. A {\em track} in $\Lambda$ is a sequence $F=(F_n)_{n\in\ZZ}$ in $\bF(\Lambda),$ such that for every $n\in\ZZ$  faces $F_n$ and $F_{n+1}$ are adjacent and, furthermore,  the edges shared by $F_n$ with its neighbors $F_{n+1}$ and $F_{n-1}$ form two opposite sides of the rhombus $F_n.$ These shared edges are referred to as the {\em ties} of the track $F$, and the rest of the edges are called the {\em rails} of the track $F$.  If all the ties of the track $F$ are deleted from   $\bE(\Lambda),$ the resulting graph has two connected components, which are called the {\em sides} of the track $F.$ Two tracks $F$ and $G$ are said to be {\em coherent} if the ties of $F$ are parallel to the ties of $G$, as shown in Figure \ref{cotra} below.

\begin{figure}[h]
 \includegraphics[width=\textwidth]{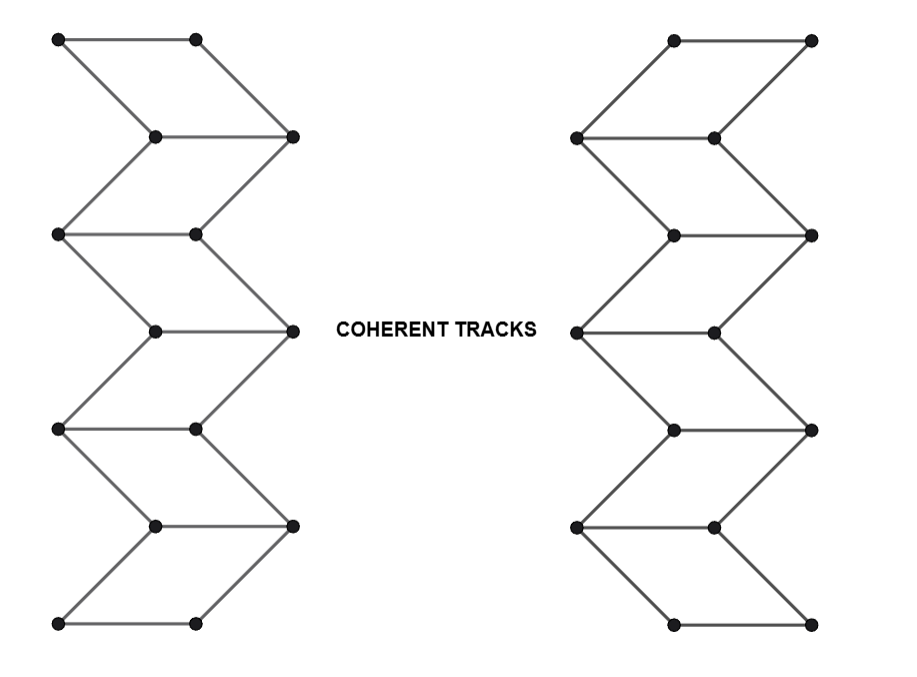}
 \caption{Coherent tracks}
 \label{cotra}
 \end{figure}

\begin{assumption}\label{mainass} Both sides of every track $F$ in $\Lambda$ contain tracks that are coherent to $F.$
\end{assumption}

From Assumption \ref{mainass}, the following proposition can be deduced immediately.

\begin{proposition}\label{Propleash} Let $e\in\hat{\bE}(\Lambda),$ and let $z\in\bV(\Lambda).$ Then there exists $N\in\NN$ and a path
$(z_0=z,z_1,\dots,z_N),$ such that $z_N-z_{N-1}=e,$ and  $z_n-z_{n-1}\not=\pm e$ if $1\leq n\leq N-1.$ 
\end{proposition}

For brevity, the path mentioned in Proposition \ref{Propleash} will be called a $e${\em-leash of } $z$ of  length $N.$

The stage is now set for the main definition of this section.
\begin{definition}\label{fwd}
The forward shift operator $Z_+:\cH(\Lambda)\longrightarrow \cH(\Lambda)$ is defined by
$$(Z_+f)(z)=\dfrac{f(0)-f(z)}{2}+\int_0^z f\delta z.$$
\end{definition}

The forward shift operator $Z_+$ is an analogue of multiplication by $z$ in the present setting.  To show this, the following technical lemma is needed.

\begin{lemma}\label{LSE}
  Let $z_0z_1z_2z_3\in\bF(\Lambda), $ and let function
$g:\{z_0,z_1,z_2,z_3\}\longrightarrow\mathcal{V}$ be such that  the Cauchy-Riemann equation 
\begin{equation}\label{CRg}
\dfrac{g(z_0)-g(z_2)}{z_0-z_2}=\dfrac{g(z_1)-g(z_3)}{z_1-z_3}
\end{equation}
holds. Furthermore, let  $w(z)$  be a non-constant linear function:
$$w(z)=\alpha z+\beta,\quad\text{where}\quad \alpha\not=0 ,\beta\in\CC.$$
\begin{enumerate}
\item If function
$f:\{z_0,z_1,z_2,z_3\}\longrightarrow\mathcal{V}$
satisfies the linear system \begin{multline}\label{beq}
	w(z_j-z_{j+1})f(z_j)-w(z_{j+1}-z_j)f(z_{j+1})=g(z_j)-g(z_{j+1}),	\\j=0,\dots, 3,\ z_4:=z_0,
	\end{multline}
then $f(z)$ satisfies the Cauchy-Riemann equation \eqref{CR}, as well.

\item If  \begin{equation}\label{nonvan1}
  \{\alpha(z_j-z_{j+1}): j=0,\dots,3\}\cap\{\pm\beta\}=\emptyset,
  \end{equation}
then, for every $\lambda\in\mathcal{V},$ 
 the linear system \eqref{beq} has a unique solution $f(z)$ normalized by
$f(z_0)=\lambda.$
\end{enumerate}
		\end{lemma}
	\begin{proof} First, if $f(z)$ satisfies \eqref{beq}, then
$$\alpha\int_{z_0z_1z_2z_3}f\delta z=\sum_{j=0}^3 \dfrac{g(z_j)-g(z_{j+1})}{2}=0,$$
and the Cauchy-Riemann equation \eqref{CR} follows.
	Next, assume that \eqref{nonvan1} is in force, so that $$w(z_j-z_{j+1})w(z_{j+1}-z_j)\not=0\quad j=0,\dots,3,$$ and	
	consider the matrix of the linear system \eqref{beq}:
	\[W=\begin{pmatrix}
	w(z_0-z_1)	&-w(z_1-z_0)	&0	&0\\
	0	&w(z_1-z_2)			&-w(z_2-z_1)	&0\\
	0	&0			&w(z_2-z_3)		&-w(z_3-z_2)\\
	-w(z_0-z_3)	&0		&0		&w(z_3-z_0)
	\end{pmatrix}.\]
	 Note that $\rank(W)\geq3.$   Since $z_0z_1z_2z_3$ is a rhombus,
	\begin{equation}\label{rhombus}\sum_{j=0}^3(-1)^jz_j=0,\end{equation} and hence
	\begin{gather*}\begin{pmatrix}
	w(z_1-z_2)&
	w(z_2-z_3)&
	w(z_3-z_0)&
	w(z_0-z_1)
	\end{pmatrix}W=0,\\
	W\begin{pmatrix} w(z_2-z_3)w(z_3-z_0)\\w(z_0-z_1)w(z_3-z_0)\\w(z_0-z_1)w(z_1-z_2)\\w(z_1-z_2)w(z_2-z_3)\end{pmatrix}=0.\end{gather*}
In particular, $\rank(W)=3,$
\begin{gather*}
\ker(W)=\spa\left\{\begin{pmatrix} w(z_2-z_3)w(z_3-z_0)\\w(z_0-z_1)w(z_3-z_0)\\w(z_0-z_1)w(z_1-z_2)\\w(z_1-z_2)w(z_2-z_3)\end{pmatrix}\right\},\\ \ker(W^*)
=\spa\left\{ \begin{pmatrix}
	w(z_1-z_2)&
	w(z_2-z_3)&
	w(z_3-z_0)&
	w(z_0-z_1)
	\end{pmatrix}^*\right\}.
	\end{gather*}
 In order to establish the existence of $f(z),$ it suffices to check that
$$ \begin{pmatrix}
	w(z_1-z_2)&
	w(z_2-z_3)&
	w(z_3-z_0)&
	w(z_0-z_1)
	\end{pmatrix}
	\begin{pmatrix} g(z_0)-g(z_1)\\
	g(z_1)-g(z_2)\\
	g(z_2)-g(z_3)\\
	g(z_3)-g(z_0)\end{pmatrix}=0.$$
However, in view of \eqref{rhombus}, the left hand side of the last equality simplifies to
$$\alpha\left((z_1-z_3)(g(z_0)-g(z_2))-(z_0-z_2)(g(z_1)-g(z_3))\right),$$
which, in view of the Cauchy-Riemann equation \eqref{CRg}, is equal to $0.$ 
	\end{proof}

 \begin{theorem} \label{fwdthm} The kernel and range of $Z_+$ can be characterized as follows:
$$\ker(Z_+)=\{0\},\quad
\ran(Z_+)=\{f\in \cH(\Lambda) : f(0)=0\}.
$$
\end{theorem}

\begin{proof}
Suppose first that $f\in\mathcal{H}(\Lambda)$ is such that $Z_+f=0$. Then 
	\[\frac{f(z)-f(0)}{2}=\int_0^zf\delta z,\quad z\in \bV(\Lambda).\]
Therefore, if $u,v\in \bV(\Lambda)$, then 
	\[f(v)-f(u)=2\int_u^v f\delta z.\]
	In particular, if the vertices $u$ and $v$ are adjacent, then
	\begin{equation}\label{keq}
	\frac{f(v)-f(u)}{v-u}=f(u)+f(v).
	\end{equation}
	Now, fix an arbitrary $z\in\bV(\Lambda)$, and let $(z=z_0,\ldots,z_N)$ be a $1$-leash of $z.$ Then (\ref{keq}) implies that
	\[f(z)=f(z_1)\frac{1+z_0-z_1}{1+z_1-z_0}=\ldots=f(z_N)\prod_{n=1}^N\frac{1+z_{n-1}-z_n}{1+z_n-z_{n-1}}=0,\]
	since $1+z_{N-1}-z_N=0$. Thus, $\ker(Z_+)=\lbrace0\rbrace.$
	As to the range of $Z_+$, the inclusion
	\[\ran(Z_+)\subset\lbrace f\in\cH(\Lambda): f(0)=0\rbrace\]
	follows immediately from the formula in Definition \ref{fwd}. To prove the opposite inclusion, assume that $g(z)$ is a DA function, such that $g(0)=0$. One needs to show that there exists a DA function $f(z)$, such that $Z_+f=g$. This can be done in a number of steps.\\
	\textbf{Step 1.} Let $f:\bV(\Lambda)\to\mathcal{V}$ be given. Then $f(z)$ is the pre-image of $g(z)$ under $Z_+$ if, and only if, the relation
	\begin{equation}\label{Deq}
	f(v)(1+u-v)-f(u)(1+v-u)=2(g(u)-g(v))	
	\end{equation}
	holds on every edge $(u,v)$ of $\Lambda$. Indeed, if $f\in\cH(\Lambda)$ is such that $Z_+f=g$, then for every pair $u,v\in\bV(\Lambda)$ it holds that
	\[g(u)-g(v)=Z_+f(u)-Z_+f(v)=\frac{f(v)-f(u)}{2}+\int_v^uf\delta z.\]
	In particular, if vertices $u$ and $v$ are adjacent, then
	\[g(u)-g(v)=\frac{f(v)-f(u)}{2}+\frac{f(v)+f(u)}{2}(u-v),\]
	which is equivalent to (\ref{Deq}). Conversely, if (\ref{Deq}) holds on every edge of $\Lambda$, then 
 $f(z)$ is a DA function by Lemma \ref{LSE} with $w(z)=1-z$ and Theorem \ref{CRthm}. Fix an arbitrary $z\in\bV(\Lambda)$ and choose a path $(z_0,\ldots,z_N)$ from $z_0=0$ to $z_N=z$. Applying (\ref{Deq}) to every edge of the path, one gets
	\begin{multline*}
	(Z_+f)(z)=\frac{f(0)-f(z)}{2}+\int_0^zf\delta z\\
	=\sum_{n=1}^N\frac{f(z_{n-1})-f(z_n)}{2}+\sum_{n=1}^N\frac{f(z_{n-1})+f(z_n)}{2}(z_n-z_{n-1})\\
		=\sum_{n=1}^N(g(z_n)-g(z_{n-1}))=g(z)-g(0)=g(z).	
		 \end{multline*}
\textbf{Step 2.} Whenever $(z,z+1)\in \bE(\Lambda)$, set
\[f(z)=g(z+1)-g(z).\]
Then (\ref{Deq}) holds on every tie and every left rail of every track with horizontal ties. To see this, note that the above equality is a special case of (\ref{Deq}) with $u=z, v=z+1$. Furthermore, if $z_0z_1z_2z_3\in\bF(\Lambda)$ is  such that $z_2=z_1+1$ and $z_3=z_0+1$, then
\begin{multline*}
(1+z_1-z_0)f(z_0)-(1+z_0-z_1)f(z_1)\\
	=(1+z_1-z_0)(g(z_0+1)-g(z_0))-(1+z_0-z_1)(g(z_1+1)-g(z_1))\\
	=(1+z_1-z_0)(g(z_0+1)-g(z_1))-(1+z_0-z_1)(g(z_1+1)-g(z_0))\\
				+2(g(z_1)-g(z_0)).
\end{multline*}
Since $g(z)$ is a DA function,  by Theorem \ref{CRthm} it satisfies the Cauchy-Riemann equation
on the face $z_0z_1z_2z_3:$
\[\frac{g(z_3)-g(z_1)}{z_3-z_1}=\frac{g(z_2)-g(z_0)}{z_2-z_0},\]
hence 
\begin{multline*}
(z_2-z_0)(g(z_3)-g(z_1))-(z_3-z_1)(g(z_2)-g(z_0))\\
=(1+z_1-z_0)(g(z_0+1)-g(z_1))-(1+z_0-z_1)(g(z_1+1)-g(z_0))=0,
\end{multline*}
and
(\ref{Deq}) holds on the edge $(z_0,z_1)$, as well.\\
\textbf{Step 3.} If $z_0z_1z_2z_3$ is a face without horizontal edges and if the value $f(z_0)$ is given, then by Lemma \ref{LSE} with $w(z)=1-z$ there is a unique way to assign the values $f(z_1), f(z_2), f(z_3)$ so that (\ref{Deq}) holds on every edge of the face $z_0z_1z_2z_3.$  Thus, one can use  relation \eqref{Deq} to extend the definition of $f(z):$ first, to the vertices with $1$-leash of length $2$ and then, by induction on the length of $1$-leash, -- to the whole of $\bV(\Lambda).$ \end{proof}

\begin{definition}\label{bwd} The backward shift operator $Z_-:\cH(\Lambda)\longrightarrow \cH(\Lambda)$ is defined as follows:
given $g\in\cH(\Lambda),$ $Z_-g$ is the unique  element of $\cH(\Lambda),$ such that
$$Z_+Z_- g(z)=g(z)-g(0).$$
\end{definition}

\begin{remark}\label{lcfwd}
Definition \ref{bwd} implies that, just like in the proof of  Theorem \ref{fwdthm}, the backward shift $f=Z_-g$ of a DA function $g(z)$ is completely determined by relation \eqref{Deq} that holds on every edge $(u,v)$ of $\Lambda.$ In particular, if $(z,z+1)\in\bE(\Lambda),$ then
$$Z_-g(z)=g(z+1)-g(z).$$
More generally,
for every $z\in\bV(\Lambda)$ there exist $N\in\mathbb{N}$, $c_1,\ldots, c_N\in\CC$ and $z_1,\ldots, z_N\in\bV(\Lambda)$ such that, for every DA function $g:\bV(\Lambda)\to\mathcal{V},$ 
	\[Z_-g(z)=\sum_{n=1}^Nc_ng(z_n).\]
\end{remark}

\begin{proposition}\label{bwdprop} Operator $Z_-$ is a left inverse of $Z_+,$ and the kernel of $Z_-$ consists of constant DA functions.
\end{proposition}
\begin{proof}
Let $f\in\cH(\Lambda)$. In view of Theorem \ref{fwdthm} and Definition \ref{bwd} one has
\begin{align*}
	Z_+Z_-Z_+f(z)&=Z_+f(z)-Z_+f(0)=Z_+f(z), \\
	&Z_-Z_+f(z)=f(z).
\end{align*}
As to the kernel of $Z_-$, it follows immediately from Definition \ref{bwd} that
\[Z_-f(z)=0\quad\Longleftrightarrow\quad f(z)-f(0)=0\quad\Longleftrightarrow\quad f(z)=f(0).\]
\end{proof}

\begin{remark}\label{contrem}
The structural identities 
\begin{equation}\label{structure}
Z_-Z_+f(z)=f(z),\quad Z_+Z_-f(z)=f(z)-f(0)\end{equation}
mirror those in the classical setting of functions analytic in an open neighborhood of the origin, with
$$Z_+f(z)=zf(z),\quad Z_-f(z)=\dfrac{f(z)-f(0)}{z}.$$
\end{remark}

\section{Eigenfunction of the backward shift and DA polynomials}\label{eigen}

In what follows, $\bS(\Lambda)$ denotes the  non-empty finite subset  of the complement of the open unit disk in the complex plane defined by
$$\bS(\Lambda)=\{t\in\CC : 2+t(1-e)=0\text{ for some }e\in \hat{\bE}(\Lambda)\}.$$ 

\begin{theorem}\label{genfunthm} \begin{enumerate}
\item Every $t\in\CC\setminus\bS(\Lambda)$ 
 is an eigenvalue of the operator $Z_-.$ The corresponding eigenspace is one-dimensional; it is spanned by the DA function $e_t(z),$ which is determined as follows: $e_t(0)=1,$ and
\begin{equation}\label{genfun}e_t(z)=
\prod_{k=1}^N\dfrac{2+t(1+z_k-z_{k-1})}{2+t(1+z_{k-1}-z_k)}\end{equation}
for $z\not=0,$ where $(z_0,z_1,\dots z_N)$ is any path from $z_0=0$ to $z_N=z.$
\item If $t\in\bS(\Lambda),$ then the operator $tI-Z_-$ is a bijection from $\cH(\Lambda)$ to itself.
\end{enumerate}
\end{theorem}
\begin{proof}
First, let $t\in\mathbb{C}\setminus \bS(\Lambda)$, and consider a function $e_t: \bV(\Lambda)\to\mathcal{V}$. Then $e_t\in\mathcal{H}(\Lambda)$ and $Z_{-}e_t=te_t$ if, and only if, the relation
\begin{equation}\label{dapropii}
e_t(u)(2+t(1+v-u))=e_t(v)(2+t(1+u-v))
\end{equation}
holds on every edge $(u,v)$ of $\Lambda$. Indeed, if $e_t(z)$ is a DA function, then, in view of the structural identities \eqref{structure}, the equality
\begin{equation}\label{eigenbwd}Z_{-}e_t(z)=te_t(z)\end{equation}
is equivalent to
\begin{equation}\label{eigenfwd}e_t(z)-e_t(0)=tZ_+e_t(z).\end{equation}
This last equality holds if, and only if, for every pair of adjacent vertices $u$ and $v$ one has
\begin{multline*}
e_t(u)-e_t(v)=t\left(Z_+e_t(u)-Z_+e_t(v)\right)=t\left(\frac{e_t(v)-e_t(u)}{2}+\int_v^ue_t\delta z\right)\\
=t\left(\frac{e_t(v)-e_t(u)}{2}+\frac{e_t(v)+e_t(u)}{2}(u-v)\right),
\end{multline*}
which is equivalent to \eqref{dapropii}. Since $t\notin\bS(\Lambda)$, it follows from Lemma \ref{LSE} with $w(z)=2+t(1-z)$ that there exists a unique  function $e_t(z)$ satisfying \eqref{dapropii} on every edge of $\Lambda$ and normalized by $e_t(0)=1.$ Moreover, this function $e_t(z)$ is DA by Theorem \ref{CRthm}. Combining relations \eqref{dapropii} for the edges of an arbitrary path $(z_0,z_1,\ldots,z_N)$  from $z_0=0$ to $z_N=z$ yields formula \eqref{genfun}.

The proof of statement (2) of the theorem is very similar to that of Theorem \ref{fwdthm}. Let $t\in\bS(\Lambda),$ then there is $s\in\hat\bE(\Lambda)$ such that $t(1-s)=-2.$ If $f\in\ker(tI-Z_-),$ then 
on every edge $(u,v)$ of $\Lambda$ it holds that
\begin{equation}\label{dapropiii}
f(u)(2+t(1+v-u))=f(v)(2+t(1+u-v)).
\end{equation}
In particular, for every edge $(u,v),$ such that $u-v=s,$ $f(v)=0.$ Given an arbitrary $z\in\bV(\Lambda),$ applying \eqref{dapropiii} to every edge of a $s$-leash of $z$ leads to $f(z)=0.$ Thus $tI-Z_-$ is injective. To show that $tI-Z_-$ is also surjective, let $g\in\cH(\Lambda).$ Then $f\in\cH(\Lambda)$ satisfies $(tI-Z_-)f=g$ if, and only if,
for every edge $(u,v)\in\bE(\Lambda)$  it holds that
\begin{equation}\label{dapropiv}
f(u)(2+t(1+v-u))-f(v)(2+t(1+u-v))=G(v)-G(u),
\end{equation}
where $G=2Z_+g.$
The existence of such a function $f(z)$ can be first 
established for every $v\in\bV(\Lambda),$ such that $(v,v+s)\in\bE(\Lambda)$ and then, via Lemma \ref{LSE}, everywhere. 
\end{proof}

\begin{remark}\label{contrem2}
In the continuous setting, one has
$$\forall t\in\CC\quad Z_-e_t(z)=te_t(z),\text{ where } e_t(z)=\dfrac{1}{1-tz}.$$
Thus the exceptional set $\bS(\Lambda)$ has no continuous counterpart.
\end{remark}

In view of \eqref{genfun}, function $e_t(z)$ is a rational function in $t$ (of degree that depends on $z$), analytic in the complement of $\bS(\Lambda)$ and, in particular, in the open unit disk $\DD.$ Consider the Taylor expansion
\begin{equation}\label{genfunexp}e_t(z)=\sum_{n=0}^\infty t^n z^{(n)},\quad t\in\DD, z\in \bV(\Lambda).\end{equation}
\begin{proposition}\label{tayprop} Taylor coefficients $z^{(n)}$ in the expansion \eqref{genfunexp} are DA functions of $z$ with the following properties:
\begin{equation}
  \label{fwdbas}z^{(0)}=1\quad\text{and}\quad z^{(n)}=Z_+z^{(n-1)},\quad n\in
  \NN,\end{equation}
in particular,  $z^{(1)}=z,$
\begin{gather}
\label{bwdbas} Z_-z^{(0)}=0\quad\text{and}\quad Z_-z^{(n)}=z^{(n-1)},\quad n\in\NN,\\
\label{cauchy}\limsup_{n\rightarrow\infty}\sqrt[n]{|z^{(n)}|}\leq 1.
\end{gather}
Moreover, $z^{(0)},z^{(1)},z^{(2)},\dots$ form a linearly independent vector family in $\cH(\Lambda).$
\end{proposition}

\begin{proof} From \eqref{genfun} it follows, in particular, that $z^{(0)}=e_0(z)=1$ and, by Proposition \ref{bwdprop}, $Z_-z^{(0)}=0.$ Comparing the coefficients of $t^n$ in the Taylor expansions on both sides of \eqref{eigenfwd}, one obtains \eqref{fwdbas} and, in view of the structural identities \eqref{structure},  \eqref{bwdbas} follows. Formula \eqref{cauchy} is the Cauchy-Hadamard estimate on the radius of convergence of the Taylor series \eqref{genfunexp}. As to the linear independence of $z^{(n)},$ observe that \eqref{fwdbas} implies, in particular,
$$0^{(n)}=0,\quad n\in\NN.$$
Therefore, if  $f(z)$ is a DA function of the form
$$f(z)=\sum_{n=0}^N\alpha_nz^{(n)},$$
where $\alpha_0,\alpha_1,\dots,\alpha_N\in\CC,$ 
then, in view of \eqref{bwdbas},
$$\alpha_n=Z_-^nf(0).$$
In particular, 
$$f=0\quad\Leftrightarrow\quad\alpha_0=\alpha_1=\dots=\alpha_N=0.$$
\end{proof}

\begin{definition} \label{dapoly} Elements of the subspace of $\cH(\Lambda)$ spanned by $$z^{(0)},z^{(1)},z^{(2)},\dots$$ are called DA polynomials. 
\end{definition}

\begin{remark} The space of DA polynomials was originally defined in \cite{rhombic}, using a different linearly independent family of DA functions. Unfortunately, it seems that the shift operators associated with the DA polynomial basis of \cite{rhombic} cannot be extended to the whole of $\cH(\Lambda).$
\end{remark}

It is useful to observe that estimate \eqref{cauchy} can be generalized as follows.

 \begin{proposition}\label{pwb}
 Let $\mathcal{V}$ be a normed vector space, and let $f:\bV(\Lambda)\longrightarrow\mathcal{V}$ be a DA function. Then, for every $z\in\bV(\Lambda),$ 
 \begin{equation}
 \label{gencau}
 \limsup_{n\rightarrow\infty}\sqrt[n]{\|Z_+^nf(z)\|}\leq 1.\end{equation}
 \end{proposition}
 \begin{proof}
 For $z=0$ \eqref{gencau} follows trivially from Theorem \ref{fwdthm}. Assume that \eqref{gencau} holds for $z=u$ and  let $\epsilon>0,$ then there exists $N\in\NN$ such that, for all $n\geq N,$
 $$\|Z_+^nf(u)\|\leq (1+\epsilon)^n.$$ 
 Furthermore, let $v\in\bV(\Lambda)$ be adjacent to $u.$ Then  (see \eqref{Deq})
 $$Z_+^{n}f(v)(1+u-v)- Z_+^{n}f(u)(1+v-u)=2(Z_+^{n+1}f(u)-Z_+^{n+1}f(v)).$$
 Since $|u-v|=1,$ it follows  that, for all $n\geq N,$
 $$\|Z_+^{n+1}f(v)\|\leq \|Z_+^nf(v)\|+(2+\epsilon)(1+\epsilon)^n,$$
 and, therefore, \begin{multline*}
 \|Z_+^nf(v)\|\leq \|Z_+^Nf(v)\|+(2+\epsilon)\dfrac{(1+\epsilon)^n-(1+\epsilon)^N}{\epsilon}
\\  \leq\left(\dfrac{\|Z_+^Nf(v)\|}{(1+\epsilon)^N}+\dfrac{2+\epsilon}{\epsilon}\right)(1+\epsilon)^n.\end{multline*}
 Thus 
 $$ \limsup_{n\rightarrow\infty}\sqrt[n]{\|Z_+^nf(z)\|}\leq 1+\epsilon.$$
 Since $\epsilon>0$ is arbitrary, \eqref{gencau} holds for $z=v,$ as well, and, by induction on the distance from $z$ to $0,$ for all $z\in\bV(\Lambda).$
 \end{proof}

One of the main obstacles in the study of DA functions is that the point-wise product of two DA functions is not a DA function, in general. One possible way to deal with this issue is to introduce 
a convolution product associated with the DA polynomial basis $z^{(n)}.$ In what follows, $\cM,\cN,\cK$ are Hilbert spaces and 
$\bL(\cM,\cN)$ denotes the Banach space of bounded linear operators from $\cM$ to $\cN.$

\begin{definition}\label{convoproddef}
The convolution product $\odot$ of a $\bL(\cM,\cN)$-valued DA polynomial 
$$p(z)=\sum_{n=0}^N A_nz^{(n)}$$
with a  $\bL(\cK,\cM)$-valued DA function $f(z)$  is given by
$$(p\odot f)(z):=\sum_{n=0}^N A_n(Z_+^nf)(z).$$
Similarly,  if $g(z)$ is a $\bL(\cN,\cK)$-valued DA function, then 
$$(g\odot p)(z):=\sum_{n=0}^N (Z_+^ng)(z)A_n.$$
\end{definition}

Note that the space of $\CC$-valued DA polynomials equipped with the convolution product $\odot$
is a commutative ring, where
$$z^{(m)}\odot z^{(n)}=z^{(m+n)}.$$

\begin{theorem}\label{resolve} Let $A\in\bL(\cM).$  Then the DA polynomial $I-zA$ is  $\odot$-invertible if,  and only if, 
\begin{equation}\label{excl}\bS(\Lambda)\cap\sigma(A)=\emptyset.\end{equation}
In this case the  $\odot$-inverse  of   $I-zA$  is given by
\begin{equation}\label{resA}(I-zA)^{-\odot}=\prod_{k=1}^N\big((2I+(1+z_k-z_{k-1})A)(2I+(1+z_{k-1}-z_k)A)^{-1}\big),\end{equation}
where $(z_0,z_1,\dots, z_N)$ is any path from $z_0=0$ to $z_N=z.$
\end{theorem}
\begin{proof}
Assume first that $I-zA$ is $\odot$-invertible. Then there is an $\bL(\cM)$-valued DA function $F(z),$ such that
$$F(z)\odot(I-zA)=(I-zA)\odot F(z)=I.$$
By definition of $\odot,$ the above equality can be re-written as
$$F(z)=I+(Z_+F)(z)A=I+A(Z_+F)(z).$$
In particular, $F(0)=I$ and
$$(Z_-F)(z)=F(z)A=AF(z).$$
 Therefore, for $t\in\CC,$
 $$(tI-Z_-)F(z)=F(z)(tI-A)=(tI-A)F(z).$$
 If $t\in\bS(\Lambda),$ then $(tI-Z_-)$ is a bijection and hence
 $$F(z)=((tI-Z_-)^{-1}F)(z)(tI-A)=(tI-A)((tI-Z_-)^{-1}F)(z).$$
 In particular,
 $$((tI-Z_-)^{-1}F)(0)(tI-A)=(tI-A)((tI-Z_-)^{-1}F)(0)=F(0)=I,$$
 and $(tI-A)$ is invertible in $\bL(\cM).$ Thus \eqref{excl} holds.
 
 The converse direction parallels the proof of Theorem \ref{genfunthm} with $A$ in place of $t.$
\end{proof}

\section{A reproducing kernel Hilbert space of DA functions}\label{RKHS-DA}
Our goal in this section is to characterize backward shift invariant reproducing kernel Hilbert spaces (RKHS) of vector-valued DA functions, adapting the ideas of \cite{ad} to the DA setting.

 Let $A\in\bL(\cM)$ be such that $\sigma(A)\cap\bS(\Lambda)=\emptyset,$ and let $C\in\bL(\cM,\cN).$ Consider the $\bL(\cM,\cN)$-valued DA function
\begin{equation}\label{absker1}F(z)=C(I-zA)^{-\odot},\quad z\in\bV(\Lambda),\end{equation}
and  the $\bL(\cN)$-valued positive kernel
\begin{equation}\label{absker2}K(z,w)=F(z)F(w)^*.\end{equation} The  following proposition is straightforward.
\begin{proposition}
The RKHS $\cH(K)$ of $\cN$-valued DA functions associated with the kernel \eqref{absker2} can be characterized as 
\begin{equation}\label{defhk}\cH(K)=\{F(z)\zeta:\zeta\in\cM\},\end{equation}
with the inner product given by
\begin{equation}\label{defhkprod}\langle F\zeta_1,F\zeta_2\rangle_{\cH(K)}=\langle(I-\pi_0)\zeta_1,\zeta_2\rangle_{\cM},\end{equation}
where $\pi_0\in\bL(\cM)$ is the orthogonal projection onto
$$\cM_0=\bigcap_{z\in\bV(\Lambda)}\ker F(z).$$
\end{proposition}

\begin{theorem}\label{bwdinvthm} $Z_-$ is a bounded operator on $\cH(K).$ Moreover, for every $t\in\bS(\Lambda)$ the operator
$(tI-Z_-)$ is invertible in $\bL(\cH(K)).$
\end{theorem}
\begin{proof} Since
$$
F(z)\odot(I-zA)=C,
$$
it follows that
$$Z_-F(z)=F(z)A.$$
Therefore, the subspace $\cM_0$ of $\cM$ is $A$-invariant, and $Z_-$ is an everywhere defined linear operator on $\cH(K).$ Furthermore, for $\zeta\in\cM$ one has
\begin{multline*}
\|Z_-F(z)\zeta\|_{\cH(K)}=\|(I-\pi_0)A\zeta\|_{\cM}=\|(I-\pi_0)A(I-\pi_0)\zeta\|_{\cM}\\
\leq\|A\|_{\bL(\cM)}\|(I-\pi_0)\zeta\|_{\cM}= \|A\|_{\bL(\cM)} \|F(z)\zeta\|_{\cH(K)},\end{multline*}
and hence $Z_-\in\bL(\cH(K)).$ Finally, if $t\in\bS(\Lambda),$ then $tI-A$ is invertible. Since $$
(tI-Z_-)F(z)=F(z)(tI-A),$$
$tI-Z_-$ is a bijection on $\cH(K)$ and hence invertible.
\end{proof}
A converse to Theorem \ref{bwdinvthm} can be formulated as follows.
\begin{theorem} \label{crt} Let $\cM$ be a RKHS of $\cN$-valued DA functions, such that $A=Z_-\in\bL(\cM)$ and $\bS(\Lambda)\cap\sigma(A)=\emptyset.$ Then, for every $f\in\cM$ and every $z\in\bV(\Lambda)$ the following point evaluation formula holds:
$$C(I-zA)^{-\odot}f=f(z),$$
where $C\in\bL(\cM,\cN)$ is given by $Cf=f(0).$ In particular, the reproducing kernel $K(z,w)$ of $\cM$ is  of the form  \eqref{absker1} -- \eqref{absker2}.
\end{theorem}

\begin{proof} Let $f\in\cM$ and  $z\in\bV(\Lambda).$ Denote
$$G(z)=(I-zA)^{-\odot}f.$$
 Then
$$
G(z)-AZ_+G(z)=f,\quad
G(0)=f,$$ and on every edge
 $(a,b)$ of $\Lambda$ the following equality holds:
 $$2(G(b)-G(a))=AG(a)(1+b-a)-AG(b)(1+a-b).$$
Now fix $z\in\bV(\Lambda)$ and choose a shortest path $(z_0,z_1,\dots,z_N)$ from $z_0=0$ to $z_N=z.$ If $N=0,$ then $z=0$ and
$$f(0)=Cf=CG(0)=C(I-zA)^{-\odot}f.$$ If $N\geq 1,$
denote
$$G(z_n)=g_n,\quad n=0,1,\dots,N.$$
Then $g_0=f$ and for every $n=1,\dots,N$ the following equalities hold:
$$
2(g_n-g_{n-1})=(1+z_n-z_{n-1})Z_-g_{n-1}-(1+z_{n-1}-z_n)Z_-g_n,$$
\begin{multline*}2Z_+(g_n-g_{n-1})=(1+z_n-z_{n-1})(g_{n-1}-g_{n-1}(0))-\\-(1+z_{n-1}-z_n)(g_n-g_n(0)).
\end{multline*}
Evaluate the functions related by the above equality at $z_{k-1}$ and $z_k$ for $k=1,\dots, N$ to conclude that
\begin{multline*}
(g_n(z_{k-1})-g_{n-1}(z_{k-1}))(1+z_k-z_{k-1})-\\-(g_n(z_k)-g_{n-1}(z_k))(1+z_{k-1}-z_k)\\=
(1+z_n-z_{n-1})(g_{n-1}(z_k)-g_{n-1}(z_{k-1}))-\\-(1+z_{n-1}-z_n)(g_n(z_k)-g_n(z_{k-1})),\end{multline*}
and hence
\begin{multline*}
(z_k-z_{k-1})(g_n(z_{k-1})-g_{n-1}(z_{k-1})+g_n(z_k)-g_{n-1}(z_k))\\=(z_n-z_{n-1})(g_{n-1}(z_k)-g_{n-1}(z_{k-1})+g_n(z_k)-g_{n}(z_{k-1})).
\end{multline*}
The above system of equations can be re-written in the matrix form as
\begin{multline}\label{mateq}E_N(I_N+Z_N-P_N)X_N(I_N-Z_N^\top -P_N)\\=(I-Z_N-P_N)X_N(I_N+Z_N^\top -P_N)E_N,\end{multline}	
where $E_N,I_N,P_N,Z_N,X_N$ are $(N+1)\times(N+1)$ matrices defined as follows:
\begin{gather*}
E_N=\diag(2,z_1-z_0,z_2-z_1,\dots,z_N-z_{N-1}),\\
I_N=\diag(1,1,\dots,1), \quad P_N=\diag(1,0,0,\dots,0),\\
Z_N=\begin{pmatrix}0&\cdots&\cdots&0\\
1&\ddots &&\vdots\\
&\ddots&\ddots&\vdots\\
0&&1&0\end{pmatrix},\quad
X_N=\begin{pmatrix} g_0(z_0)&\cdots &g_N(z_0)\\
\vdots&&\vdots\\
g_0(z_N)&\cdots&g_N(z_N)\end{pmatrix}.
\end{gather*}
It can be shown by induction on $N$ that every solution $X_N$ of \eqref{mateq} is symmetric: $X_N=X_N^\top .$ Indeed, for $N=1$ \eqref{mateq} reduces to
$$\begin{pmatrix} 1& 1\end{pmatrix}X_1\begin{pmatrix}-1\\ 1\end{pmatrix}=\begin{pmatrix} -1& 1\end{pmatrix}X_1\begin{pmatrix}1\\ 1\end{pmatrix},$$
and the result follows by straightforward computation. For $N>1$ one may assume that $X_{N-1}$ is symmetric and, without loss of generality, that $X_N$ is skew-symmetric: $X_N^\top =-X_N.$ Then $X_N$ is of the form
$$X_N=\begin{pmatrix}0_{N\times N}&v\\-v^\top &0\end{pmatrix},$$ where $v$ is a vector-column of dimension $N,$ and
 equation \eqref{mateq} reduces to a system of two equations:
$$\left\{\begin{aligned} v^\top j+j^\top v&=0,\\
E_{N-1}(I_{N-1}+Z_{N-1}-P_{N-1})v&=(z_N-z_{N-1})(I_{N-1}-Z_{N-1}-P_{N-1})v,
\end{aligned}\right.$$
where
$$j=\begin{pmatrix}0&0&\cdots&0&1\end{pmatrix}^\top .$$
Since the path from $z_0$ to $z_N$ is of minimal length, the diagonal matrix
$E_{N-1}+(z_N-z_{N-1})I_{N-1}$ is invertible, and the second equation can be re-written as
\begin{multline*}Z_{N-1}v\\=-(E_{N-1}+(z_N-z_{N-1})I_{N-1})^{-1}(E_{N-1}-(z_N-z_{N-1})I_{N-1})(I-P_{N-1})v.\end{multline*}
Since $v^\top j=j^\top v=0,$ $v=0$ and $X_N=X_N^\top ,$ which means
$$f(z)=g_0(z_N)=g_N(z_0)=Cg_N=CG(z_N)=CG(z)=C(I-zA)^{-\odot}f.$$
\end{proof}

\section{Hardy space and Schur class}\label{Sec-Sch}
\begin{definition} Let $\mu>1.$ The weighted Hardy space $\bH^2_\mu(\cN)$ is the RKHS of $\cN$-valued DA functions on $\Lambda$ with the reproducing kernel
$$K_\mu(z,w)=\sum_{n=0}^\infty\dfrac{z^{(n)}\overline{w^{(n)}}}{\mu^{2n}}I_{\cN}.$$
\end{definition}
The following two observations are straightforward; the proofs are omitted.
\begin{proposition} The weighted Hardy space $\bH^2_\mu(\cN)$ can be characterized as follows:
$$\bH^2_\mu(\cN)=\{f(z)=\sum_{n=0}^\infty z^{(n)}\hat{f}(n):\sum_{n=0}^\infty \mu^{2n}\|\hat{f}(n)\|_\cN^2<\infty\}.$$
\end{proposition}
\begin{proposition}\label{isomshift} The forward and backward shifts $Z_+,Z_-$ are bounded operators on $\bH^2_\mu(\cN).$ Moreover,  
$Z_+^*=\mu^2Z_-,$  $\frac{1}{\mu}Z_+$ is an isometry, $\mu Z_-$ is a co-isometry.
\end{proposition}
\begin{definition}
Let $\cH$ be a Hilbert space, and let $\left(\begin{smallmatrix} A& B\\C&D\end{smallmatrix}\right)$ be a co-isometry from $\left(\begin{smallmatrix}\cH\\ \cM\end{smallmatrix}\right)$ to $\left(\begin{smallmatrix}\cH\\ \cN\end{smallmatrix}\right).$ For $\mu>1,$ the DA function
\begin{equation}\label{defSch}
S_\mu(z)=D+C\left(I-z\dfrac{A}{\mu}\right)^{-\odot}\odot\left(z\dfrac{B}{\mu}\right)\end{equation}
is called the $\mu$-dilated $\bL(\cM,\cN)$-valued DA Schur function associated with the co-isometry  $\left(\begin{smallmatrix} A& B\\C&D\end{smallmatrix}\right).$\end{definition}

\begin{theorem}\label{Schkerthm} If  $\left(\begin{smallmatrix} A& B\\C&D\end{smallmatrix}\right)$ is a co-isometry from $\left(\begin{smallmatrix}\cH\\ \cM\end{smallmatrix}\right)$ to $\left(\begin{smallmatrix}\cH\\ \cN\end{smallmatrix}\right),$ then for every $\mu>1$ the kernel
\begin{equation}\label{kersch}
K^{S_\mu}(z,w)=\sum_{n=0}^\infty\dfrac{z^{(n)}\overline{w^{(n)}}I_{\cN}-(Z_+^nS_\mu)(z)(Z_+^nS_\mu)(w)^*}{\mu^{2n}},
\end{equation}
where $S_\mu$ is the associated $\mu$-dilated DA Schur function,
 is of the form
$$K^{S_\mu}(z,w)=F(z)F(w)^*,$$
with
$$F(z)=C\left(I-z\dfrac{A}{\mu}\right)^{-\odot}.$$
In particular, kernel \eqref{kersch} is positive in this case.

Conversely, if  $\mu>1$ and a $\bL(\cM,\cN)$-valued DA function $S_\mu(z)$ are such that  kernel \eqref{kersch} is positive, then there exist a Hilbert space $\cH$ and a co-isometry $\left(\begin{smallmatrix} A& B\\C&D\end{smallmatrix}\right)$ from $\left(\begin{smallmatrix}\cH\\ \cM\end{smallmatrix}\right)$ to $\left(\begin{smallmatrix}\cH\\ \cN\end{smallmatrix}\right),$ such that
 $S_\mu(z)$ is the  associated $\mu$-dilated DA Schur function.
\end{theorem}

\begin{proof} The idea of the proof is similar to the $\epsilon$-method of Kre\u{\i}n
and Langer (see e.g. \cite{krein}), that was adapted in \cite{adrs} using the underlying
reproducing kernel Pontryagin spaces. 
First, assume that $$
\begin{pmatrix} A& B\\C&D\end{pmatrix}:\begin{pmatrix}\cH\\ \cM\end{pmatrix}\longrightarrow\begin{pmatrix}\cH\\ \cN\end{pmatrix}$$ is a co-isometry, that $\mu>1,$
and that $S_\mu(z)$ is given by \eqref{defSch}. 
Then 
\begin{gather*}
\dfrac{1}{\mu}\begin{pmatrix} Z_+F(z) & \mu I \end{pmatrix}\begin{pmatrix} B\\ D\end{pmatrix}=S_\mu(z),\\[1ex]
\dfrac{1}{\mu}\begin{pmatrix} Z_+F(z) & \mu I \end{pmatrix}\begin{pmatrix} A \\ C\end{pmatrix}=F(z).
\end{gather*}
 Therefore, for every $n\in\NN\cup\{0\},$
 \begin{multline*}\dfrac{ 
 z^{(n)} \overline{ w^{(n)} } I - (Z_+^nS_\mu)(z) (Z_+^nS_\mu)(w)^*
 } {\mu^{2n}}
 \\=\dfrac{
 (Z_+^nF)(z)(Z_+^nF)(w)^*
 }{\mu^{2n}}
 -\dfrac{
 (Z_+^{n+1}F)(z)(Z_+^{n+1}F)(w)^*}{\mu^{2n+2}},
 \end{multline*}
 and, in view of Proposition \ref{pwb},
 $$K^{S_\mu}(z,w)=F(z)F(w)^*.$$
 
 To prove the converse statement, assume that
 $\mu>1$ and a $\bL(\cM,\cN)$-valued DA function $S_\mu(z)$ are such that  kernel \eqref{kersch} is positive.
 Consider the associated RKHS $\cH=\cH(K^{S_\mu})$ and denote
 $$K^{S_\mu}_w(z)=K^{S_\mu}(z,w),\quad L^{S_\mu}_w(z)=(Z_-K^{S_\mu}_z)(w)^*.$$
 Note that, in view of Remark \ref{lcfwd}, for every $w\in\bV(\Lambda)$ and every $\zeta\in\cN,$ function $L^{S_\mu}_w(z)\zeta$ is an element of the RKHS $\cH$ and that
 \begin{equation}\label{secfor11} L^{S_\mu}_w(z)=-S_\mu(z)(Z_-S_\mu)(w)^*+\dfrac{1}{\mu^2}(Z_+K^{S_\mu}_w)(z).\end{equation}
 Next, define a linear relation $\mathcal{R}$ in
 $$\begin{pmatrix} \cH\\ \cN\end{pmatrix}\times\begin{pmatrix} \cH\\ \cM\end{pmatrix}$$ as follows:
 $$\mathcal{R}=\spa\left\{\left( \begin{pmatrix} K^{S_\mu}_w\alpha\\ \beta\end{pmatrix},
 \begin{pmatrix} \mu L^{S_\mu}_w\alpha+K^{S_\mu}_0\beta\\
\mu (Z_-S_\mu)(w)^*\alpha+S_\mu(0)^*\beta\end{pmatrix}\right)\right\},$$
 where $w$ ranges over $\bV(\Lambda),$ and $\alpha,\beta$ - over $\cN.$
One can check that relation $\mathcal{R}$ is isometric in a number of steps.
Firstly, for every $w\in\bV(\Lambda)$ and every $\alpha,\beta\in\cN,$ it holds that
\begin{multline*} 
 \langle L^{S_\mu}_w\alpha,K^{S_\mu}_0\beta\rangle_{\cH}=\langle L^{S_\mu}_w(0)\alpha,\beta\rangle_\cN=\langle  (Z_-K^{S_\mu}_0)(w)^*\alpha,\beta\rangle_\cN
 \\
 		=-\langle S_\mu(0)(Z_-S_\mu)(w)^*\alpha,\beta\rangle_\cN
 		=-\langle(Z_-S_\mu)(w)^*\alpha,S_\mu(0)^*\beta\rangle_\cN.
 \end{multline*}	
 Secondly,	 for every $\beta_1,\beta_2\in\cN$ one has
 \begin{multline*}		
 		\langle K^{S_\mu}_0\beta_1, K^{S_\mu}_0\beta_2\rangle_{\cH}=\langle K^{S\mu}_0(0)\beta_1,\beta_2\rangle_\cN\\
 		=\langle \beta_1,\beta_2\rangle_\cN-\langle S_\mu(0)^*\beta_1,S_\mu(0)^*\beta_2\rangle_\cM.
\end{multline*} 
Thirdly,
according to  Remark \ref{lcfwd}, for every $w_2\in\bV(\Lambda)$ there exist $N\in\mathbb{N}$, $c_1,\ldots, c_N\in\mathbb{C}$ and $ z_1,\ldots,z_N\in\bV(\Lambda)$ such that, for every DA function $f(z)$, it holds that $$Z_-f(w_2)=\sum_{j=1}^Nc_jf(z_j).$$ 
Accordingly,
$$L^{S_\mu}_{w_2}(z)=\sum_{n=1}^N\bar{c}_nK^{S_\mu}_{z_n}(z),$$
and, for every $w_1\in\bV(\Lambda$ and every $\alpha_1,\alpha_2\in\cN,$ one has
\begin{multline*}
 		\langle\mu L^{S_\mu}_{w_1}\alpha_1,\mu L^{S_\mu}_{w_2}\alpha_2\rangle_{\cH}
 		=\mu^2\sum_{n=1}^N \langle L^{S_\mu}_{w_1}\alpha_1,\bar{c_n}K^{S_\mu}_{z_n}\alpha_2\rangle_{\cH}\\
 		=\mu^2\sum_{n=1}^Nc_n\langle L^{S_\mu}_{w_1}(z_n)\alpha_1,\alpha_2\rangle_{\cN}
 		=\mu^2\langle (Z_-L^{S_\mu}_{w_1})(w_2)\alpha_1,\alpha_2\rangle_{\cN}.
\end{multline*}
Finally, in view of \eqref{secfor11}, this last expression can be re-written as
\begin{multline*}\langle K^{S_\mu}_{w_1}(w_2)\alpha_1,\alpha_2\rangle_{\cN}-\mu^2
		\langle (Z_-S_\mu)(w_2)(Z_-S_\mu)(w_1)^*\alpha_1,\alpha_2\rangle_{\cN}
		\\
 		=\langle K^{S_\mu}_{w_1}\alpha_1,K^{S_\mu}_{w_2}\alpha_2\rangle_{\cH}-\langle\mu(Z_-S_\mu)(w_1)^*\alpha_1,\mu(Z_-S_\mu)(w_2)^*\alpha_2\rangle_\cM.
\end{multline*}
 Since the domain of $\mathcal{R}$ is dense in $\left(\begin{smallmatrix} \cH\\ \cN\end{smallmatrix}\right).$ $\mathcal{R}$ extends as the graph of an isometry from $\left(\begin{smallmatrix} \cH\\ \cN\end{smallmatrix}\right)$ to $\left(\begin{smallmatrix} \cH\\ \cM\end{smallmatrix}\right).$The adjoint is a coisometry
 $$\begin{pmatrix} A& B\\C&D\end{pmatrix}:\begin{pmatrix}\cH\\ \cM\end{pmatrix}\longrightarrow\begin{pmatrix}\cH\\ \cN\end{pmatrix},$$
 where
 \begin{gather*}
 Af(z)=\mu (Z_-f)(z),\quad B\zeta=\mu (Z_-S_\mu)(z)\zeta,\\
 Cf(z)=f(0),\quad D\zeta=S_\mu(0)\zeta.
 \end{gather*}
 In particular, $Z_-=\frac{A}{\mu}\in\bL(\cH)$ is a strict contraction, so that $\sigma(Z_-)\cap\bS(\lambda)=\emptyset.$ By Theorem \ref{crt},  it holds for every $f\in\cH$ that
 $$C\left(I-z\dfrac{A}{\mu}\right)^{-\odot}f=f(z).$$
 Therefore, for $\zeta\in\cM$ one has
 \begin{gather*}C\left(I-z\dfrac{A}{\mu}\right)^{-\odot}B\zeta=\mu(Z_-S_\mu)(z)\zeta,\\
 Z_+C\left(I-z\dfrac{A}{\mu}\right)^{-\odot}B\zeta=\mu(S_\mu(z)-S_\mu(0))\zeta=\mu S_\mu(z)-\mu D\zeta,\\
 S_\mu(z)\zeta=\left(D+\dfrac{1}{\mu}Z_+C\left(I-z\dfrac{A}{\mu}\right)^{-\odot}B\right)\zeta,
 \end{gather*}
 which is the same as \eqref{defSch}.
\end{proof}

Positivity of the kernel \eqref{kersch} allows to extend the convolution product to the Schur class.
\begin{theorem} Let $S_\mu(z)$ be a $\bL(\cM,\cN)$-valued DA function, and let $\mu>1.$ The associated kernel \eqref{kersch} is positive if, and only if
the operator $M_{S_\mu}$ defined for 
$$f(z)=\sum_{n=0}^\infty z^{(n)}\hat{f}(n)\in\bH^2_\mu(\cM)$$
by
$$M_{S_\mu}f(z)=S_\mu\odot f(z)=\sum_{n=0}^\infty  Z_+^nS_\mu(z)\hat{f}(n)$$
is a contraction in  $\bL(\bH^2_\mu(\cM),\bH^2_\mu(\cN)).$
\end{theorem}

\begin{proof} Assume first that kernel \eqref{kersch} is positive.
Let $w\in\bV(\Lambda),$ let $\zeta\in\cN,$ and for $f(z)=K_\mu(z,w)\zeta\in\bH^2_\mu(\cN)$ define operator $Cf(z)$ by
$$Cf(z)=\sum_{n=0}^\infty z^{(n)}\dfrac{Z_+^nS_\mu(w)^*\zeta}{\mu^{2n}}.$$
In view of Proposition \ref{pwb},   $Cf(z)\in\bH^2_\mu(\cM).$ Moreover, the positivity of kernel \eqref{kersch} implies that $C$ extends as a linear contraction first on the dense linear subset
$$\spa\{K_\mu(\cdot,w)\zeta:w\in\bV(\Lambda),\zeta\in\cN\},$$
of $\bH^2_\mu(\cN)$ and then -- as a contraction in $\bL(\bH^2_\mu(\cN),\bH^2_\mu(\cM)).$ The adjoint of $C$ is the multiplication operator $M_{S_\mu}.$

Conversely, assume that $M_{S_\mu}$ is contractive. Then, for every $\zeta\in\cN$ and $w\in\bV(\Lambda),$
$$K^{S_\mu}(\cdot,w)\zeta=(I-M_{S_\mu}M_{S_\mu}^*)K_\mu(\cdot,w)\zeta,$$
and hence $K^{S_\mu}(z,w)$ is a positive kernel.
\end{proof}

\section{Blaschke factors}\label{Sec-Bla}
\begin{definition} Let $a\in\CC,\mu\in\RR$ be such that $|a|\leq 1$ and $\mu>1.$  The  $\mu$-dilated Blaschke factor $B^a_\mu(z)$ is a $\mu$-dilated DA Schur function associated with the $2\times 2$ unitary matrix
$$\begin{pmatrix} A&B\\C&D\end{pmatrix}=\begin{pmatrix}\bar{a}/\mu &\sqrt{1-|a|^2/\mu^2}\\
\sqrt{1-|a|^2/\mu^2}&-a/\mu\end{pmatrix}.$$
\end{definition}
\begin{proposition} The multiplication operator $M_{B_\mu^a}$ is an isometry on $\bH^2_\mu(\CC).$
\end{proposition}
\begin{proof}
Note first that, by Proposition \ref{isomshift}, one has
\begin{gather}\label{mb}
M_{B_\mu^a}=\dfrac{1}{\mu}\left(-aI+\left(1-\dfrac{|a|^2}{\mu^2}\right)Z_+\left(I-\dfrac{\bar{a}}{\mu^2}Z_+\right)^{-1}\right),\\
\label{mb*}
M_{B_\mu^a}^*=\dfrac{1}{\mu}\left(-\bar{a}I+\mu^2\left(1-\dfrac{|a|^2}{\mu^2}\right)Z_-\left(I-aZ_-\right)^{-1}\right).
\end{gather}
The fact that $M_{B_\mu^a}^*M_{B_\mu^a}=I$ follows immediately from the identity
\begin{multline*}\left(1-\dfrac{|a|^2}{\mu^2}\right)\left(I-aZ_-\right)^{-1}\left(I-\dfrac{\bar{a}}{\mu^2}Z_+\right)^{-1}\\=I+aZ_-\left(I-aZ_-\right)^{-1}+\dfrac{\bar{a}}{\mu^2}Z_+\left(I-\dfrac{\bar{a}}{\mu^2}Z_+\right)^{-1}.\end{multline*}
\end{proof}
\begin{proposition} \label{mbe}The isometry $M_{B_\mu^a}$ satisfies the following relation:
$$I-M_{B_\mu^a}M_{B_\mu^a}^*=\left(1-\dfrac{|a|^2}{\mu^2}\right)\left(I-\dfrac{\bar{a}}{\mu^2}Z_+\right)^{-1}P_0\left(I-aZ_-\right)^{-1},$$
where $P_0$ is the orthogonal projection in $\bH^2_\mu(\CC)$ onto the subspace of constant functions. In particular,
$$\ker(M_{B_\mu^a}^*)=\ran(I-M_{B_\mu^a}M_{B_\mu^a}^*)=\spa\{e_{\bar{a}/\mu^2}(\cdot)\}.$$
\end{proposition}
\begin{proof}
The result  follows from \eqref{mb}, \eqref{mb*}, the fact that
$$Z_+Z_-=I-P_0,$$
and the identity
\begin{multline*}
\left(1-\dfrac{|a|^2}{\mu^2}\right)\left(I-\dfrac{\bar{a}}{\mu^2}Z_+\right)^{-1}\left(I-aZ_-\right)^{-1}\\=I+aZ_-\left(I-aZ_-\right)^{-1}+\dfrac{\bar{a}}{\mu^2}Z_+\left(I-\dfrac{\bar{a}}{\mu^2}Z_+\right)^{-1}
-\\ -\dfrac{|a|^2}{\mu^2}\left(I-\dfrac{\bar{a}}{\mu^2}Z_+\right)^{-1}P_0\left(I-aZ_-\right)^{-1}.
\end{multline*}
\end{proof}

 Let  $\gamma=(z_0=0,\dots,z_N)$ be a path in $\Lambda.$ We assume that $\gamma$ has minimal possible length -- that is, the graph distance from $z_0$ to $z_N$ equals precisely $N$ -- so that the family of DA functions
$$\{K_\mu(\cdot,z_n):n=0,\dots N\}$$
is linearly independent. For $n=0,\dots,N$ denote by $\cM_n$  the finite-dimensional subspace of $\bH^2_\mu(\CC),$ spanned by $\{K_\mu(\cdot,z_k):k=0,\dots,n\}.$

\begin{proposition}\label{mninv}
The subspace $\cM_n$ is invariant under both $Z_-$ and $M_{B_\mu^a}^*.$
\end{proposition}
\begin{proof}
$Z_-$-invariance of $\cM_n$ follows from Proposition \ref{isomshift} and the observation that the orthogonal complement of $\cM_n,$ 
$$\cM_n^\perp=\{f\in\bH_\mu^2(\CC):f(z_0)=f(z_1)=\dots=f(z_n)=0\},$$
is $Z_+$-invariant.  $M_{B_\mu^a}^*$-invariance then follows from \eqref{mb*}.
\end{proof}

In the following proposition a particular value of $a$ is chosen.

\begin{proposition} \label{intprop} Let $1\leq n\leq N,$ and let
$$a_n=\dfrac{z_n-z_{n-1}-1}{2}.$$ The following relations hold:
\begin{gather*}\left(Z_- -\dfrac{\bar{a}_n}{\mu^2} I\right)K(\cdot,z_n)\in \cM_{n-1},\\
\cM_n^\perp\subset\ker(I-M_{B_\mu^{a_n}}M_{B_\mu^{a_n}}^*),\\
M_{B_\mu^{a_n}}^*\cM_n=\cM_{n-1},\\
M_{B_\mu^{a_n}}\cM_{n-1}\subset\cM_n,\quad M_{B_\mu^{a_n}}^*\cM_n^\perp\subset\cM_{n-1}^\perp,\\
M_{B_\mu^{a_n}}\cM_{n-1}^\perp=\cM_n^\perp.
\end{gather*}
\end{proposition}
\begin{proof}
According to Proposition \ref{isomshift},
$$Z_-K_\mu(z,z_n)=\dfrac{1}{\mu^2}\overline{Z_+K_\mu(z_n,z)},$$
where both shift operators act on $K_\mu$ as a DA function in the first variable. But then it follows from the formula in Definition \ref{fwd} that
$$Z_-K_\mu(\cdot,z_n)-\dfrac{\bar{a}_n}{\mu^2}K_\mu(\cdot,z_n)\in\spa\{K_\mu(\cdot,z_k):k=0,\dots,n-1\}=\cM_{n-1}.$$
Therefore, as follows from Proposition \ref{mninv} and formula \eqref{mb*},
$$
M_{B_\mu^{a_n}}^*K(\cdot,z_n)\in\cM_{n-1},\quad M_{B_\mu^{a_n}}^*\cM_n\subset\cM_{n-1},
\quad  M_{B_\mu^{a_n}}\cM_{n-1}^\perp\subset \cM_{n}^\perp.$$
Moreover, since $\dim(\cM_{n-1})<\dim(\cM_n)$ and
$$\left(Z_- - \dfrac{\bar{a}_n}{\mu^2} I\right)\cM_n\subset\cM_{n-1},$$
the $Z_-$-invariant subspace $\cM_n$ contains an eigenvector of $Z_-$ corresponding to the eigenvalue $\bar{a}/\mu^2.$ 
Therefore, $e_{\bar{a}_n/\mu^2}(\cdot)\in\cM_n.$ In view of Proposition \ref{mbe}, this implies 
\begin{gather*}
\ker M_{B_\mu^{a_n}}^*=\ran(I-M_{B_\mu^{a_n}}M_{B_\mu^{a_n}}^*)\subset\cM_n,\\
\cM_n^\perp\subset\ker(I-M_{B_\mu^{a_n}}M_{B_\mu^{a_n}}^*)=\ran M_{B_\mu^{a_n}}.
\end{gather*}
Furthermore, since $\dim(\ker(M_{B_\mu^{a_n}}^*))=1,$ $$\dim(M_{B_\mu^{a_n}}^*\cM_n)=
\dim(\cM_{n})-1=\dim(\cM_{n-1}),$$
hence
$$
M_{B_\mu^{a_n}}^*\cM_n=\cM_{n-1}.$$
It follows that
$$M_{B_\mu^{a_n}}\cM_{n-1}=M_{B_\mu^{a_n}}M_{B_\mu^{a_n}}^*\cM_n\subset\cM_n\text{ and }M_{B_\mu^{a_n}}^*\cM_{n}^\perp\subset\cM_{n-1}^\perp.$$
Finally,
since $\cM_n^\perp\subset\ker(I-M_{B_\mu^a}M_{B_\mu^a}^*),$
$$\cM_n^\perp=M_{B_\mu^{a_n}}M_{B_\mu^{a_n}}^*\cM_n^\perp\subset M_{B_\mu^{a_n}}\cM_{n-1}^\perp.$$ The opposite inclusion has already been established.
\end{proof}
Iterating the last formula in Proposition \ref{intprop} for $n=1,\dots,N$ yields the following theorem concerning basic interpolation in the weighted Hardy space $\bH^2_\mu(\CC).$
\begin{theorem}\label{interpolthm}
Let $\lambda\in\bV(\Lambda),$ and let $(z_0,z_1,\dots,z_N)$ be a shortest path from $z_0=0$ to $z_N=\lambda.$ Then $f\in\bH_\mu^2(\CC)$ is a solution of the basic interpolation problem
$$f(z_0)=f(z_1)=\dots=f(z_N)=0$$
if, and only if, $f(z)$ is of the form
$$f(z)=\dfrac{z}{\mu}\odot B_\mu^{a_1}(z)\odot \dots\odot B_\mu^{a_N}(z)\odot g(z),$$
where 
$$a_n=\dfrac{z_n-z_{n-1}-1}{2},\quad n=1,\dots N,$$
and $g\in\bH^2_\mu(\CC).$ In this case
$$\|f\|_{\bH^2_\mu(\CC)}=\|g\|_{\bH^2_\mu(\CC)}.$$
\end{theorem}

\section*{Acknowledgments}
Daniel Alpay thanks the Foster G. and Mary McGaw Professorship in Mathematical Sciences, which supported this research. \\

\section*{Data availability statement} There is no data associated with this paper.
\bibliographystyle{plain}

\end{document}